\documentclass[10pt,oneside,reqno]{amsart}

  \usepackage[utf8]{inputenc}
  \usepackage[T1]{fontenc}
  \usepackage{lmodern}

  \usepackage{hyperref}

  \usepackage{amsmath}
  \usepackage{amsthm}
  \usepackage{amsfonts}
  \usepackage{amssymb}
  \usepackage{amscd}
  \usepackage{amsbsy}

  \usepackage{pdfpages}
  \usepackage{fancyhdr}
  \usepackage{geometry}
  \usepackage{setspace}
  \usepackage{xcolor}
  \usepackage{pifont}

  \usepackage{graphicx}
  \usepackage{tabularx}
  \usepackage{epsfig}
  \usepackage[all]{xy}
  \usepackage{tikz}
  \usetikzlibrary{matrix}

  \usepackage{fixmath}

  \usepackage{appendix}

  \usepackage{enumerate}

  \usepackage[sort, numbers]{natbib}
  \usepackage{bibentry}

  \newtheorem{theorem}{Theorem}[section]

  \newtheorem{lemma}[theorem]{Lemma}
  
  \newtheorem{claim}[theorem]{Claim}
  \newtheorem{corollary}[theorem]{Corollary}
  
  \theoremstyle{definition}
  \newtheorem{definition}[theorem]{Definition}

  \newtheorem*{remark}{Remark}

  \newtheorem{conjecture}[theorem]{Conjecture}

  \numberwithin{equation}{section}

  \newcommand\restr[2]{{
  \left.\kern-\nulldelimiterspace 
  #1 
  \littletaller 
  \right|_{#2} 
  }}

  \allowdisplaybreaks

\title{Minimizing edge-length polyhedrons}

\author[A. Valfells]{Asgeir Valfells}
\address{Department of Mathematics, Rice University, Houston, TX~77005, USA}
\email{Asgeir@rice.edu}
\thanks{}

\begin{document}
\begin{abstract}
    A 1957 conjecture by Zdzislaw Melzak, that the unit volume polyhedron with least edge length was a triangular right prism, with edge length $2^{2/3}3^{11/6}$. We present a variety of necessary local criteria for any minimizer. In the case that we are restricted to convex polyhedrons we demonstrate that all vertices must be of degree three, the number of triangular faces is at most 14, and we describe the behavior of quadrilateral faces should they become arbitrarily small.
\end{abstract}

\maketitle
\section{Introduction}
\subsection{The Essential History}

In 1965 Zdzislaw Melzak conjectured that the convex unit volume polyhedron with least edge length was a triangular right prism, with edge length $2^{2/3}3^{11/6}\approx 11.896$\cite{melzak_qs}. To this day it is not known whether a minimizer exists. In the years since (and in fact preceding) progress on the problem has been sparse, surprising for a problem that could have been posed in antiquity. The current results come from two perspectives, variational arguments and compactness arguments.\\

What's known in the bona fide polyhedral case can be done with straightforward variational calculations. Melzak himself demonstrated that the regular tetrahedron was optimal among tetrahedrons\cite{melzak_tetrahedrons} and Berger\cite{berger_2002} showed that the triangular prism was optimal among the class of prisms, regular polyhedrons, and pyramids. More exciting is Berger's approach to the convex problem via compactness arguments. Admitting generalized candidates, which compactify the space of convex polyhedrons, he demonstrates the existence of a minimizer in that broader space with accompanying regularity results. This work was later extended to higher dimension by Scott \cite{RyanScott}.

\subsection{Methods and Results}

In the vein of Melzak we'll approach the problem from a variational perspective. In the writings of Melzak and Berger the polyhedrons are viewed as a collection of vertices which are then perturbed. This is a problem in the generic case since most perturbations would mean the perturbed vertex is no longer co-planar with the vertices that it shared faces with. The immediate remedy is to modify the adjacent faces (as illustrated in) in a way that introduces a new edge, not ideal when we are trying to minimize edge length.\\

In this chapter we will rely on what could broadly be described as perturbations of the faces. We found that this strategy was also implemented by Besicovitch and Eggleston \cite{Besicovitch} in tackling a polyhedral edge length minimizing problem where the constraint was not volume, but rather that the polyhedron must be able to contain a unit ball. \\   

To make the problem slightly more tractable we'll make a cosmetic change. We'll seek to minimize the cube of the edge length and normalize our functional w.r.t. volume so any polyhedron is admissible. Letting $e(P)$ be the edge length of a polyhedron $P$ and $v(P)$ the volume we are trying to minimize the functional $m(P)=e(P)^3/v(P)$. Notice $m$, in this text referred to as the Melzak ratio, is scale invariant. Throughout the text we'll talk about "a Melzak problem" as one where we want to minimize the Melzak ratio among some class of polyhedral objects, whether it be all polyhedrons, generalized polyhedrons, polyhedrons with restrictions on the number of faces, convex polyhedrons, or any other number of geometric restrictions.\\

Taking a perturbation of some polyhedron $P_0$, $\mathcal{P}=\{P_t\}_{t \in \left[0,\epsilon \right]}$ we can define functions $E_{\mathcal{P}}(t)$, $V_{\mathcal{P}}(t)$, $M_{\mathcal{P}}(t)$ to be $e(P_t),v(P_t),m(P_t)$ respectively. Omitting writing $\mathcal{P}$ and assuming both $E$ and $V$ are differentiable note:

$$M'(t)=\frac{3E(t)^2}{V(t)}E'(t)-\frac{E(t)^3}{V(t)^2}V'(t)$$

A minimizer $P_0$, in any given class, cannot admit a perturbation with $M'(0)<0$. Our goal will be to extract local criteria on minimizing candidates by choice of $\mathcal{P}$. Our strongest results will be for convex polyhedrons so the following sequence will be helpful.

\begin{claim}
\label{MinimizingSequence}
There exists a sequence of convex polyhedrons $P_n$ that:\\
i) Minimizes the Melzak ratio among convex polyhedrons\\
ii) If $P_n \neq P_{n+1}$ (up to scaling) then $P_n$ has fewer faces\\
iii) $P_n$ minimizes Melzak ratio among convex polyhedrons with equally many faces.
\end{claim}
\begin{proof}
First let $P_1$ be the regular tetrahedron. Known to Melzak \cite{melzak_tetrahedrons} to minimize the Melzak ratio among all tetrahedrons. To get $P_{n+1}$ we first choose the polyhedron, $P$, with the fewest sides, $k$, so $P$ lower Melzak ratio than $P_n$. Since the space of unit volume polyhedrons with $k$ faces and some bounded edge length is compact we let $P_{n+1}$ be the polyhedron among them that minimizes the Melzak ratio.    
\end{proof}

\subsection{Summary of Results}

The ambitious goals, to find or demonstrate the existence of a minimizer to some formulation of Melzak's problem is not achieved. However we do find local criteria for minimizers. The most approachable formulation is for the case of convex polyhedrons. For the elements of \ref{MinimizingSequence} we are able to demonstrate several criteria. The most complete are in \ref{ExposedFace}, \ref{FiniteTriangles}, and \ref{QuadProperty}. Respectively we are able to say that all vertices must be of degree 3, we are able to bound the number of triangular faces, and we are able to describe necessary behavior for small quadrilateral faces.\\

While we don't get any global data for the Melzak problem admitting all polyhedrons we do describe local data, the results described above are generalized as they can be to polyhedrons that are not necessarily convex. On top of generalizing the local criteria of convex polyhedrons to the non-convex we also show that the vertex degree on potential minimizers is bounded by the curvature at the vertex for locally convex vertices.

\section{Vertex Degree Bounds}

\begin{definition}
A vertex $H$ of $P$ is exposed it is contained in an open neighborhood $U_H$ s.t. $U_H\cap P$ is convex.
\end{definition}

\begin{definition}
A face $F$ of $P$ is exposed it is contained in an open neighborhood $U_F$ s.t. $U_F\cap P$ is convex. Equivalently a face $F$ of $P$ is exposed if all its vertices are exposed.
\end{definition}

\subsection{Exposed Faces}
\begin{theorem}
\label{ExposedFace}
If $P$ has an exposed face $F$ with a vertex of degree more than 3 we can perturb $P$ to decrease the Melzak ratio.
\end{theorem}

\begin{proof}
Our strategy is to create two perturbations and deduce the result from there.
Our first perturbation, $\mathcal{P}$, could be qualitatively described  as sliding an exposed face in the outward direction. More precisely, for an exposed face $F$ we can pick $U_F$ so $U\cap P_0$ is an intersection of closed (exc. one open) half spaces. By translating the halfspace that corresponds to the face $F$ in the outward normal direction by $t$ and then replacing $U_F\cap P$ with this new intersection we get $P_t$, taking care to have $t$ sufficiently small. We'll call corresponding face $F_t$.\\

Finding $M'(0)$ calls for two observations. First, for each vertex $H$ of $F_0$ let $H_t$ be the corresponding vertex of $F_t$. When $t>0$ we immediately see that $H_t$ is of degree 3, with unit vectors $u_{1,t,H}$ and $u_{2,t,H}$ along the edges of $F_t$ and vector $v_{H}$ s.t. the vector between $H_t$ and $H_s$ is $(t-s)v_{H}$. We can then attribute $E_{\mathcal{P}}'(t)$ to the vertices where it is contributed.
$$ E'_{\mathcal{P}}(t)=\sum_H E'_{\mathcal{P},H}(t)= \sum_H <v_{H},u_{1,H,t}+u_{2,H,t}>+||v_H||$$.

Next, taking $t$ to $0$ we see that the change in volume approaches a prism with base $F_0$ and height $t$, so $V'(0)$ is equal to the area of $F_0$. Taken together
$$M'_{\mathcal{P}}(0)=\frac{3E(0)^2}{V(0)}\left(\sum_H ||v_H||-<v_{H},u_{1,H,0}+u_{2,H,0}>\right)-\frac{E(0)^3}{V(0)^2}A(F_0) $$

Now we construct a second perturbation $\mathcal{Q}$, where instead of translating the halfspace corresponding to $F_0$ into the outward facing direction we translate it into the inward facing direction.
Here we are not guaranteed a one to one correspondence between the vertices of $F_0$ and the vertices of $F_t$. In fact, if $H$ is degree $k$ then $F_t$ has $k-2$ corresponding vertices. Let $H_{\cdot,t}$ be these corresponding vertices, with $v_{H,\cdot}$ the vectors corresponding to $-v_{H}$ in the earlier perturbation. As in the earlier perturbation we find $V'_{\mathcal{Q}}(0)=-A(F_0)$.\\

If all vertices of degree 3 then we get $M'\mathcal{P}(0)=-M'\mathcal{Q}(0)$, following the calculation of $M'\mathcal{P}(0)$. However, let one $H$ have degree $k$. Then we can attribute the following rate of change in edge length to $H$ and evaluate at $t=0$:
$$E'_{\mathcal{Q},H}(0)=<v_{H_1},u_{1,H,0}>+<v_{H_{k-2}},u_{2,H,0}>+\sum_{n=1}^{k-2}-||v_{H_n}||+\sum_{n=1}^{k-3}||v_{H_n}-v_{H_{n+1}}||$$

Now we want to show  $E'_{\mathcal{Q},H}(0)<-E'_{\mathcal{P},H}(0)$. To demonstrate this we'll use a small geometric trick, illustrated in figure \ref{Geotrick}. Note labels $a,b$ denote edge lengths, $a$ the dotted edges and $b$ solid ones, and are new:
\begin{center}
\includegraphics[scale=1.25]{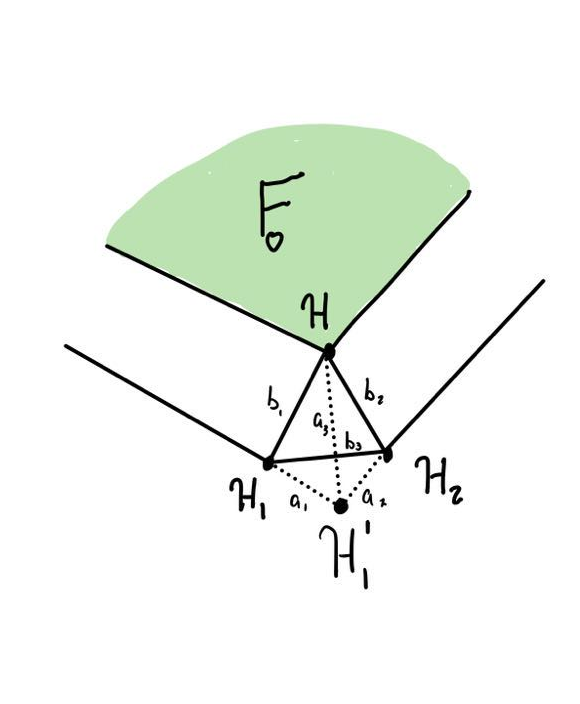}    \label{Geotrick}
\end{center}

Again considering $U_F$ to be an intersection of halfspaces we observe that removing the halfspace along the face $H,H_{1,t},H_{2,t}$ induces a new vertex labeled $H'_{1,t}$. Reindexing $H_{n,t}=H'_{n-1,t}$ for all $n>3$ we get:

$$<v_{H'_1},u_{1,H,0}>+<v_{H'_{k-2}},u_{2,H,0}>+\sum_{n=1}^{k-3}-||v_{H'_n}||+\sum_{n=1}^{k-4}||v_{H'_n}-v_{H'_{n+1}}||=$$ $$<v_{H,1},u_{1,H,0}>+<v_{H,k-2},u_{2,H,0}>+\sum_{n=1}^{k-2}-||v_{H,n}||+\sum_{n=1}^{k-3}||v_{H,n}-v_{H,n+1}||+a_1+a_2+b_1+b_2-a_3-b_3$$
The triangle inequality then tells us $a_1+a_2+b_1+b_2-a_3-b_3>0$. Iterating this process gives us a single vertex $H'_t$, namely $H_t-tv_{H}$. All taken together we get:

$$-E'_{\mathcal{P},H}(0) = <v_{H'_1},u_{1,H,0}>+<v_{H'_{k-2}},u_{2,H,0}> - ||v_{H'}|| = $$ $$<v_{H,1},u_{1,H,0}>+<v_{H,k-2},u_{2,H,0}>+\sum_{n=1}^{k-2}-||v_{H,n}||+\sum_{n=1}^{k-3}||v_{H,n}-v_{H,n+1}||+\sum a_{\cdot} + \sum b_{\cdot}>$$ $$< v_{H,1},u_{1,H,0}>+<v_{H,k-2},u_{2,H,0}>+\sum_{n=1}^{k-2}-||v_{H,n}||+\sum_{n=1}^{k-3}||v_{H,n}-v_{H,n+1}|| = E'_{\mathcal{Q},H}(0)$$ Giving our desired outcome $M'_{\mathcal{Q}}(0)<-M'_{\mathcal{P}}(0)$.\\

Now, either $M'_{\mathcal{P}}(0)< 0$ or $M'_{\mathcal{Q}}(0)< 0$. Interestingly our perturbation does not change the number of faces on $P$ though it may change the number of vertices on $F$. 
\end{proof}
While being able to prescribe this local condition is certainly nice in the Melzak problem for polyhedrons, it is most applicable in the convex Melzak problem.
\begin{corollary}
All elements of the sequence \ref{MinimizingSequence} have vertices only of degree 3.
\end{corollary}

In fact, it was too strong a condition to ask that the face be exposed. We can weaken our conditions and generalize \ref{ExposedFace} slightly.
\begin{theorem}
\label{SemiExposedFace}
If $P$ has a face $F$ all but $2$ adjacent vertices exposed and with an exposed vertex of degree more than 3 we can perturb $P$ to decrease the Melzak ratio.
\end{theorem}
\begin{proof}
The idea of the proof is the same as in \ref{ExposedFace}, however our perturbation rotates the face about the edge between the two vertices which are not necessarily exposed, with $\mathcal{P}$ rotating the face outward and $\mathcal{Q}$ rotating the face inward. There is only one small technical difference, which does not alter the argument proving \ref{ExposedFace} at all. 
$$v_H=\lim_{t\rightarrow 0^+} \frac{H_t-H}{t}$$
\end{proof}

\subsection{Exposed Vertices}

Should we limit ourselves to knowing a vertex is exposed we can no longer apply any perturbations on the faces of the polyhedron. What we will instead do could be qualitatively described as shaving off some infinitesimally small piece of the corner. Before starting we'll introduce some tools and observations from discrete differential geometry.

\begin{definition}
    If $H$ is a vertex of $P$ then the Gauss image at the vertex $H$ is given by sending each face to their respective outward facing normal vector in $S^2$ and drawing geodesics between points corresponding to adjacent faces.
\end{definition}
 \begin{remark}
     If $P$ is convex then taking the Gauss image of each vertex of $P$ gives us a dual-graph of $P$ on $S^2$.
 \end{remark}
   \begin{remark}
       The discretization of Gaussian curvature is the angle deficit at a vertex, that is $2\pi - \sum \alpha_i$ where $\alpha$ are the angles of adjacent faces at the vertex. This corresponds to the algebraic area of the Gauss image at the vertex, a classic result in discrete differential geometry in the convex case \cite{Alexandrov2005} that also holds in the general case \cite{AngleDeficitAll}.
   \end{remark}
   \begin{remark}
      The length of the geodesics gives us $\pi$ minus the dihedral angle between the (oriented) faces. The distance from a point to a geodesic thus corresponds similarly to the angle between the corresponding edge and face.
   \end{remark}

\begin{definition}
    If $H$ is exposed then $g(H)$ is a convex spherical polyhedron. Taking the inscribed circle in $g(H)$ we can determine the incenter, $c_H \in S^2$, and inradius, $\theta_H \in (0,2\pi)$ of $g(H)$.
\end{definition}

\begin{theorem}\label{ExposedVertex}
    If $P$ has an exposed vertex $H$ such that $deg(H)>\frac{2\pi}{tan(\theta)}$ then $P$ can be perturbed to have lower Melzak ratio and one more face.
\end{theorem}
\begin{proof}
Similar to our construction in \ref{ExposedFace} we take $U_H \cap P$ and intersect it with a halfspace normal to $c_H$ intersecting with $H$ and translate it inward by $t$, we then replace $U_H\cap P$ with this intersection to get $P_t$.

  This construction specifically tells us that at $P_t$ we have vertices $H_{i,t}$ corresponding to each side of $g(H)$ with the length of the vector $v_{H_i}$ given by $t/sin(\theta_i)$ where $\theta_i$ is the distance from the corresponding edge of $g(H)$ to $c_H$. It's quick to see that under this perturbation we get $\Delta |V(t)|<ct^2$ so $V'(0)=0$ and in total:
  $$M'(0)=\frac{3E(0)^2}{V(0)}\left(\sum_i ||v_{H_i}-v_{H_{i+1}}||-||v_{H_i}||\right)$$ Finally we observe that:
$$\sum_i ||v_{H_i}-v_{H_{i+1}}||<\frac{2\pi}{\tan(\theta_H)} \text{ and } \sum_i ||v_{H_i}||>deg(H)-2+\frac{2}{\sin{\theta_H}}$$.\\ By hypothesis $deg(H)>\frac{2\pi}{\tan(\theta_H)}$ so we get our desired result.
\end{proof}
Should we want to rephrase this in terms of the angle deficit at $H$ we can use the following lemma.
\begin{lemma}
    The area of $g(H)$ is bounded above and below by $\theta_H$, respectively $4\theta_H$ and $2\pi(1-cos(\theta_H))$.
\end{lemma}
\begin{proof} The first bound is obvious as the area within the incircle must be less than the area of $g(H)$, specifically $\text{Area}(g(H)) > 2\pi (1-cos(\theta_H))$.
The second bound relies on the fact that either the incircle is tangent to two geodesics on opposite sides or to at least three geodesics. In the former case $\text{Area}(g(H)) \leq 4\theta_H$ attaining equality when inscribed in a bigon. In the latter case we take $\alpha,\beta$ as illustrated and using the spherical cosine rule calculate up to a constant:
\begin{center}
\includegraphics[scale=0.75]{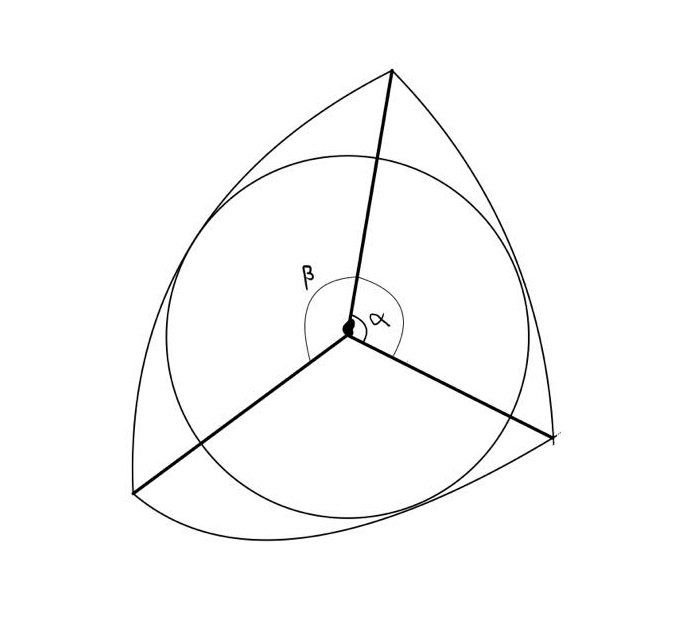}
\end{center}
$$\text{Area}(g(H)) = 2\left(arccos(sin(\alpha/2)cos(\theta_H))+arccos(sin((\beta-\alpha)/2)cos(\theta_H))\right)+C$$
Taking the derivative in $\alpha$ gives us:
$$\frac{d}{d\alpha}\text{Area}(g(H)) = \frac{cos(\theta_H)cos(\frac{\beta-\alpha}{2})}{2\sqrt{1-cos^2(\theta_H)sin^2(\frac{\beta-\alpha}{2})}}-\frac{cos(\theta_H)cos(\frac{\alpha}{2})}{2\sqrt{1-cos^2(\theta_H)sin^2(\frac{\alpha}{2})}}$$
Which we can verify is negative if $\alpha<\beta/2$. Thus the area is maximized when $g(H)$ approaches a bigon.
\end{proof}
\begin{corollary}
    Letting $\alpha$ be the angle deficit at $H$ then $\ref{ExposedVertex}$ holds under the condition that $deg(H)>\frac{1-\frac{\alpha}{2\pi}}{\sqrt{1-(1-\frac{\alpha}{2\pi})^2}}$.
\end{corollary}
\begin{proof}
    We know $\alpha > 2\pi (1-cos(\theta_H))$ and since $arccos$ is decreasing we also know $\theta_H < arccos(1-\frac{\alpha}{2\pi})$. Again $\frac{2\pi}{tan(\cdot)}$ is decreasing so we conclude that:
    $$\frac{2\pi}{tan(\theta_H)} > \frac{2\pi}{tan(arccos(1-\frac{\alpha}{2\pi}))} >  \frac{2\pi-\alpha}{\sqrt{1-(1-\frac{\alpha}{2\pi})^2}}$$
\end{proof}
Again there is a slight generalization to our conditions that we can make. Specifically if $P^C$ has an exposed vertex at $H$ we can get similar results. To use more comfortable language we'll say vertex is \emph{negatively exposed}.
\begin{lemma}\label{FirstNegative}
If $P$ has a negatively exposed vertex $H$ with $deg(H)>\frac{2\pi}{tan(\theta_{H_C})}$ we can perturb $P$ to get a lower Melzak ratio and one more face.
\end{lemma}
\begin{proof}
We start by taking $U_H\cap P^C$ and again map the adjacent faces of $P^C$ to $S^2$ via the Gauss map. As before we take the center of the incenter $c_H$, and intersect $U_H \cap P^C$ with a halfspace normal to $c_H$ and translated inward by $t$ to get $R_t$. Finally get $P_t$ by replacing $(U_H\cap P)$ with $U_H\setminus R_t$.\\
As in \ref{ExposedVertex} $V'(0)=0$ and $E'(0)$ will be defined the exact same way so we can again get our desired result by having $deg(H)$ large enough.
\end{proof}
\begin{remark}
We can in fact carry over a proof of \ref{SemiExposedFace} (and implicitly \ref{ExposedFace}) directly by inverting the perturbations as we did in the proof of \ref{FirstNegative}. We'll omit the explicit construction and calculations, but suffice it to say the broadest generalization we get is the following.
\end{remark}
\begin{theorem}\label{BroadestGeneralization}
In addendum to \ref{SemiExposedFace},\ref{ExposedVertex}, and \ref{FirstNegative}, if $P$ has a face with all but two adjacent vertices negatively exposed then either all but possibly those two must be of degree 3 or we can perturb $P$ to improve the Melzak ratio.
\end{theorem}

\section{Geometric Bounds}
In the last section we got a handle on local criteria for the vertex degree of minimizing candidates, in this one we attempt to do something similar for the faces of candidates. The first case we will work on is an exposed triangle, then a dihedral bound in the convex problem, and finally several quadrilateral cases.

\begin{remark}
For any polyhedron we're aware that the faces of high degree are bounded by the faces of low degree, specifically  $\sum_F (deg(F)-6)\leq -6\chi(P)$. Successfully bounding the low-degree faces up to 6 (in any class) would allow us to easily demonstrate that there exists a minimizer. Even more limited bounding may give us obstructions to constructing competitors to the triangular prism.
\end{remark}

\subsection{Triangular Faces}
We recall that on a minimizing candidate an exposed face it must have vertices only of degree $3$. To help understand our next perturbation we'll imagine sliding a vertex H (negatively) in the direction of the third edge which does not lay on our triangular face. If we think about each of the perturbations as moving vertices of the triangle we see they are all linear combinations of sliding one vertex along the induced vector. Concretely, we'll use the same perturbation as described for $\mathcal{P}$ in \ref{SemiExposedFace}, fixing all but one vertex $H_1$.
\begin{theorem}
\label{FirstDeficitBound}
If $P$ has an exposed triangular face $F$ with the sum of angle deficits at the vertices of $F$ (here the angle deficit associated with $F$) less than or equal to $\pi/2$ we can perturb $P$ to another polyhedron with lesser Melzak ratio and the same number of faces. 
\end{theorem}
\begin{proof}
 Letting $\gamma_{1,2}$ be the angle between $v_{H_1}$ and the line from $H_1$ to $H_2$ and using the $\mathcal{P}$ described above we get:

$$V'(0)>0\text{ and }E'(0)=||v_{H_1}||-<v_{H_{1}},u_{1,H_1,0}+u_{2,H_1,0}>=||v_{H_{1}}||(1-\cos(\gamma_{1,2})-\cos(\gamma_{1,3}))$$
Should $M'(0)\geq 0$ we need $1>\cos(\gamma_{1,2})+\cos(\gamma_{1,3}))$. Taking the sum all $H_{\cdot}$ gives us $$3>\sum_{i,j \in\{ 1,2,3\}} \cos(\gamma_{i,j})$$
Now we can observe that the angle deficit associated with $F$ is: $$\sum_{i,j \in\{ 1,2,3\}}\gamma_{i,j}-\pi>\frac{\pi}{2}$$
With the inequality algebraically tight as $\gamma$ approach alternating $0$ and $\pi/2$ (Note: this is not a geometric possibility so the inequality is far from sharp). Thus, if the associated angle deficit is less than or equal to $\pi/2$ one of the aforementioned $M$ satisfy $M'(0)<0$.
\end{proof}
Note in this proof we also got a small condition the exposed vertices of triangular faces that $1>\cos(\gamma_{1,2})+\cos(\gamma_{1,3}))$. Unfortunately we cannot make this any stronger as $V'(0)$ may become arbitrarily small on a small enough triangular face.\\

Taken together with \ref{ExposedFace} we can make a finiteness claim about a Melzak problem.

\begin{corollary}\label{FiniteTriangles}
If $P$ is in the sequence from \ref{MinimizingSequence} then $P$ has at most $14$ triangular faces.
\end{corollary}

\begin{proof}
 When $P$ is a polyhedron all vertices must be of degree $3$, therefore the only case where three triangular faces share a vertex is the tetrahedron. The polyhedral Gauss-Bonnet theorem \cite{GaussBonnet} tells us that the angle deficit at all vertices adjacent with a triangle (in fact all vertices) is $4\pi$ so \ref{FirstDeficitBound} tells us that there are at most $30$ vertices adjacent to triangular faces, thus at most $15$ triangles. Noticing that then there is one triangle which is not adjacent to another, thus contributing vertices with angle deficit greater than $\pi/2$, we tighten our bound down to $14$ triangular faces.\end{proof}

Unsurprisingly we can say something similar about negatively exposed triangles. While we won't get any helpful corollaries we would be remiss not to mention it.

\begin{claim}
If $P$ has a negatively exposed triangular face $F$ with the sum of angle deficits at the vertices of $F$ greater than or equal to $\pi$ we can perturb $P$ to another polyhedron with lesser Melzak ratio and the same number of faces. 
\end{claim}
\begin{proof}
As we did for the exposed triangle we note that all vertices must be of degree $3$, then we'll apply the perturbation analogous to \ref{SemiExposedFace} along any edge. To say this explicitly, as in \ref{FirstNegative}: taking $U_F\cap P^C$ as an intersection of halfspaces we take the halfspace along $F$ and rotate it away from $P$ along a designated edge of the triangle and again take the complement. Replacing $U_F \cap P$, with our rotation appropriately chosen depending on $t$, gives us a perturbation $\mathcal{P}$ with the following properties:
$$V'(0)>0 \text{ and } E'(0)=-||v_{H_1}|| -<v_{H_{1}},u_{1,H_1,0}+u_{2,H_1,0}>=||v_{H_{1}}||(-1-\cos(\gamma_{1,2})-\cos(\gamma_{1,3}))$$
Here $\gamma_{1,2}$ is the angle between $v_{H_1}$ and the line from $H_1$ to $H_2$. Should $M'(0)$ always be positive we need $E'(0)$ to be positive. As in \ref{FirstNegative} we sum over all vertices to find the angle deficit:

$$\sum_{i,j \in\{ 1,2,3\}}-cos(\gamma_{i,j})>3 \text{ and } \sum_{i,j \in\{ 1,2,3\}}\gamma_{i,j}-\pi<\pi$$

With the inequality tight as all $\gamma$ approach $\pi/3$. Thus if the associated angle deficit is greater than or equal to $\pi$ one of the aforementioned $M$ satisfy $M'(0)<0$.
\end{proof}
It's worth noticing that when manipulating exposed triangular faces we were able to perturb a single vertex to induce a face, of course any three vertices lay on a plane. This is unfortunately not the case for faces of higher degree, our perturbations are more restrictive and the geometric conditions we are able to impose on quadrilaterals are significantly weaker.\\ 

\subsection{Quadrilateral Faces}

We'll state at the forefront that our first dihedral angle bound is nothing special and geometrically obvious, however it will come up during our treatment of quadrilateral faces. In semblance with our other bounds we'll start by proving an easy dihedral angle bound and then generalize to a slightly stronger one. The first dihedral angle bound we'll discuss is one that is geometrically clear; if $P$ is a convex polyhedron with small Melzak ratio then two adjacent faces cannot have a small dihedral angle between them. We will however use the existence of a generalized dihedral angle bound so it is prudent to offer a proof.\\

To introduce some vocabulary we'll use along the way, if $P$ is a convex polyhedron then $P$ can be written as the (non-trivial) intersection of some halfspaces $\bigcap S_i$, we'll call the associated planes the limiting planes of $P$ at face $F_i$.

\begin{claim}
    \label{DihedralBound}If $P$ is convex with $M(P)\leq B$ then the dihedral angle between two adjacent faces of $P$ is at least $2arctan(\frac{27}{4B^3})$.
\end{claim}
\begin{proof}
Since $P$ is convex we can take the limiting planes of two adjacent faces with minimal dihedral angle, $\alpha$, between them. By choice of coordinates we can let the intersection between these two planes lie on $\mathbb{R}\times\{0\}\times \{0\}$ and let $\{0\}\times\mathbb{R}\times\{0\}$ be the orthogonal axis that bisects the dihedral angle. No we normalize s.t. $P$ has volume $1$ and let $W_1,W_2,W_3$ be the maximal differences among points of $P$ in each of those axes.  The following facts are immediately clear:\\
\begin{equation}
    \frac{W_1W_2W_3}{2} \geq 1 \qquad \frac{W_3}{2W_2}\leq tan\left(\frac{\alpha}{2}\right) \qquad W_1+W_2 \leq E(P)
\nonumber \end{equation}
Taken together give $W_1 W_2^2 tan\left(\frac{\alpha}{2}\right) \geq 1$ and $W_1+W_2 \leq B$. The l.h.s. of the former inequality is maximized when $W_2 = \frac{2}{3} B $ and $W_1 = \frac{1}{3}B$, giving $B^3 tan\left(\frac{\alpha}{2}\right) \geq \frac{27}{4}$.
\end{proof}
Note that for any elements of the minimizing sequence detailed in \ref{MinimizingSequence} we get that a dihedral angle lower bound of $2arctan(\frac{1}{8\cdot 3^{5/2}}) \approx 0.016$, nothing illuminating, but importantly positive.\\

The more important dihedral angle bound is one where the two faces are not adjacent, but merely close. To justify it a little; the possible event we want to control is if our minimizing sequence doesn't find a fixed point. That would call for the number of faces going to infinity, the size of the faces would decay rapidly (see \cite{berger_2002} for more details) and these small faces would have neighboring faces that were close to one another. Broadening the dihedral angle bound therefore gives us a control on the dihedral angle between all limiting planes surrounding almost all faces in the scenario we want to control.

\begin{lemma}
\label{HelpfulDihedral} Let $P$ be a convex polyhedron with $M(P)\leq B$. Taking two limiting planes of $P$ with faces at most $d$ apart along a shared face, the dihedral angle between the planes is bounded from below by a non-zero constant increasing in $d$ and depending on $B$.
\end{lemma}

\begin{proof}
  The proof starts the same, organizing the axes and defining our $W_\cdot$ except we'll define $W'_2$ the distance from $\mathbb{R}\times\{0\}\times\mathbb{R}$ to $P$. Here we see that $W'_2 \leq \frac{d}{2tan\left(\frac{\alpha}{2}\right)}$. The analogous inequalities to those used in \ref{DihedralBound} are the following:
  \begin{equation}
    W_1W_2W_3 \geq 1 \qquad \frac{W_3}{2(W_2+W'_2)}\leq tan\left(\frac{\alpha}{2}\right) \qquad W_1+W_2 \leq E(P)
\nonumber \end{equation}
Rearranging as earlier gives us $(W_1W_2^2+W_1W_2W'_2)tan\left(\frac{\alpha}{2}\right) \geq \frac{1}{2}$ and $W_1+W_2 \leq B$. Optimizing each summand with the constraint of the latter inequality then gives us:
$$ \frac{4}{27}B^3tan\left(\frac{\alpha}{2}\right)\geq \frac{1}{2}-\frac{1}{4}dB^2$$
From here one can isolate $\alpha \geq C_{B,d}$ as desired.
\end{proof}

From here we are able to work towards our first, and simplest quadrilateral bound, on rectangular faces of convex polyhedrons. In principle we'll apply a similar perturbation as we did in the proof of \ref{FirstDeficitBound}, fixing two adjacent vertices and rotating the limiting plane away from the polyhedron, and again noting that any perturbation that trivially increases volume is a sum of such perturbations. There are a handful of complications here, first and foremost the fact that an exposed quadrilateral face is not determined uniquely by the image of neighboring faces under the Gauss map.\\

  First let's extract the critical part of the proof \ref{FirstDeficitBound} and apply it to quadrilateral faces. This is the condition to get non-increasing edge-length under the perturbation described above. A helpful observation is that if $P=\bigcap S_i$ is convex with a rectangular face $F_j$ then $P^*_j = \bigcap_{i\neq j} S_i$ is $P$ union a wedge. We'll call this the protruding polyhedron, $P'_j$, (or here protruding wedge) at $F_j$.

  \begin{claim}\label{QuadPerturbCond}Let $P$ be a polyhedron with exposed quadrilateral face $F$ w. vertices $H_\cdot$ and vector $v_{H_i}$ as the non-base edge from $H_i$ in the protruding polyhedron at $F$. We can perturb $P$ to get a small Melzak ratio if the following is non-positive:
$$ ||v_{H_1}||-<v_{H_{1}},u_{1,H_1}+u_{2,H_1}>+||v_{H_1}||-<v_{H_{2}},u_{1,H_1}+u_{2,H_2}>$$
  \end{claim}

  \begin{lemma}\label{WedgeBound}
      Let $P'$ be the wedge with all acute dihedral angles along some rectangular base (call this a good wedge) and longest edge along the base of length $1$. Also give $P'$ a lower bound on dihedral angle. The curvature at the vertices away from the base is bounded from below by a nonzero constant depending only on the dihedral angle bound and the height of $P'$. 
  \end{lemma}
  \begin{proof}
     Fix a height for $P'$, we will show that the spherical quadrilateral induced by the image of the four limiting planes of $P'$ away from the base cannot be arbitrarily small. While we'll proceed with the proof more qualitatively it'll be helpful to name the faces of the wedge $F'_\cdot$, associated edges along the base $L'_\cdot$, and the two top vertices $H'_\cdot$ as illustrated.\\
     \begin{center}
\includegraphics[scale=1]{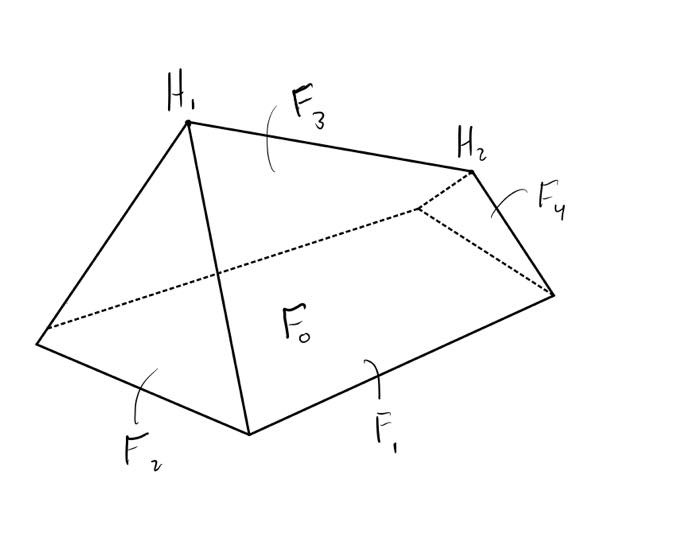}    \label{WedgeLabels}
\end{center}
     Let us describe the image of the Gauss map in spherical coordinates with $-g(F'_0)$ as the zenith. The fixed height bounds the polar angle of $g(F'_i)$ from below for $i\neq 4$. If $g(H'_1)$ is to be arbitrarily small this means that these $g(F'_i)$ must either all be arbitrarily close to one another in polar angle, or one pair arbitrarily close in azimuthal angle and the third arbitrarily close in azimuthal angle up to a shift of $\pi$ .\\

     In the former case by asking $g(H'_1)+g(H'_2)$ to be small enough we can ensure that $L'_2$ and $L'_4$ are as close to parallel as we wish. Applying the dihedral angle bound to $F'_2$ and $F'_4$ then creates an obstruction to the existence of such a quadrilateral face.\\

     The latter case is marginally more delicate. If $g(F'_2)$ is the distant lone vertex this reduces to the former case. Otherwise the fact that an edge of $g(H'_1)$ can be chosen as close to the zenith as we wish means that the edge between $H'_1$ and $H'_2$ can have as low a slope as we wish, so the height of $H'_2$ is bounded from below and thus the polar angle of $g(F'_4)$ is bounded from below so it must be close to another vertex in azimuthal angle. This calls for two subcases.\\

     If there are two pairs arbitrarily close in azimuthal angle then by asking $g(H'_1)+g(H'_2)$ to be small enough we can ensure that $L'_2$ and $L'_4$ are as close to parallel as we wish. Applying the dihedral angle bound to $F'_2$ and $F'_4$ then creates an obstruction to the existence of such a quadrilateral face.\\

     If $g(F'_1),g(F'_2),g(F'_4)$ can be chosen arbitrarily close in azimuthal angle then applying same strategy to $F'_1$ and $F'_3$ creates an obstruction to the existence of such a quadrilateral face. 
  \end{proof}

  Note the previous lemma is almost certainly already known, or perhaps entirely obvious to those the geometrically inclined, but we did not see it in the introductory texts so we prove it here. Next we want to analyze these protruding polyhedrons. We are happy with any finite bound on the number of these faces on a minimizer to a convex Melzak problem so we limit ourselves to considering the protruding polyhedrons at faces with some arbitrarily small size and associated curvature.

  \begin{corollary} \
      For any given $h,B\in \mathbb{R}$ there exists $C_{h,B}\in \mathbb{N}$ such that for any convex polyhedron $M(P)<B$ then $P$ has at most $C_{h,B}$ quadrilateral faces where the height of the protruding normalized good wedge is greater than $h$.
  \end{corollary}

\begin{proof}
    Follows directly from \ref{HelpfulDihedral} and \ref{WedgeBound}.
\end{proof}

The strategy we used for triangular faces does not prove as fruitful for quadrilateral faces. Instead we will take quadrilateral faces along of the minimizing sequence \ref{MinimizingSequence}, normalize their protruding wedge, and consider a limit of a converging subsequence. From here we'll extract data about sufficiently small quadrilateral faces. Interestingly enough this strategy isn't fruitful for triangular faces.\\

Specifically, take a subsequence of \ref{MinimizingSequence} with an increasing number of quadrilateral faces with good protruding wedges. Construct a sequence of faces by choosing one from each of these polyhedrons s.t. their size and associated curvature converge to $0$. Normalizing their protruding wedge to have the longest base edge of length $1$. The space of normalized wedges is compact so there exists a convergent subsequence with limit $P'$. Calling the convergent subsequence of wedges $P'_i$ we construct associated sequences of polyhedrons $P_i$ (a subsequence of \ref{MinimizingSequence}) and the associated faces $F_i$. 

\begin{theorem}\label{QuadProperty}
    Let the number of quadrilateral faces with good protruding wedges on elements of \ref{MinimizingSequence} be unbounded. We then know the following:\\
    (i) $P'$ is degenerate and of height $0$.\\
    (ii) The expression \ref{QuadPerturbCond} is equal to $0$ for all edges of $P'$.\\
    (iii) $P'$ is a pyramid.\\
\end{theorem}
\begin{proof}
    Statement (i) follows directly from \ref{WedgeBound}.
    
    For notational convenience we will denote the expression as $R(\cdot',L)$, that is the negative instantaneous change in edge length of a protruding wedge when perturbed along edge $L$ as described in \ref{QuadPerturbCond}. It is immediately clear by construction of \ref{MinimizingSequence} that the $R(P'_i,\cdot)$ must be positive, by continuity it is non-negative for any edge of $P'$.

    If $R(P',L)>0$ is positive then for large enough $i$ we have $R(P'_i,L)>R(P',L)/2$. Now apply the inverse perturbation in the to $P_i$ at $F_i$, recall from earlier:

    $$M'(0)=\frac{3E(0)^2}{V(0)}E'(0)-\frac{E(0)^3}{V(t)^2}V'(0)$$

    Here $E'(0)$ is negative and scales w.r.t. the diameter of $F_i$ while $V'(0)$ is proportionate to the diameter of $F_i$ squared. For small enough a face this contradicts the construction of \ref{MinimizingSequence}, proving statement (ii).\\

    For statement (iii) we apply statement (ii). For large enough $i$ then  $R(P'_i,L)$ becomes arbitrarily small, so replacing $P_i$ with $P^*_i$ decreases edge length on the order of $|H'_1 - H'_2|$ times the diameter of $F_i$. The construction of \ref{MinimizingSequence} then assures us that $H'_1 = H'_2$.
\end{proof}

\begin{remark} \label{CleanCond}
    We know $P'$ is a degenerate pyramid we can rephrase the conditions from (ii). View the base vertices of $P'$ as four vectors $p_i\in \mathbb{R}^2$ and the top vertex at $0$.
$$F(i):=|p_i|\left(1 - \frac{<p_i - p_{i+1},p_i>}{|p_i||p_{i+1}-p_i|} - \frac{ <p_i -p_{i-1},p_i>}{|p_i||p_{i-1}-p_i|}\right)$$
Letting $p_0 = p_4$ we can write (ii) as $F(1)=-F(2)=F(3)=-F(4)$.
\end{remark}

\begin{corollary}
There number of quadrilateral faces with good protruding wedges and two adjacent angles of less than $\pi/2-c$ on elements of \ref{MinimizingSequence} is bounded for any $c$.
\end{corollary}
\begin{proof}
    Should this be the case we can pick $P'$ to have two adjacent acute angles, contradicting \ref{CleanCond} 
\end{proof}
\begin{corollary}
    The number of quadrilateral faces with good protruding wedges and two adjacent angles of greater than $\pi/3$ on elements of \ref{MinimizingSequence} is bounded.
\end{corollary}
\begin{proof}
    As above.
\end{proof}
While we can certainly apply the strategy of the triangular faces to a host of special cases of quadrilateral faces (e.g. appropriate approximately rectangular faces) it does not bring us significantly closer to eliminating the possibility of an unbounded number of quadrilateral faces. We leave this with a conjecture that is almost certainly easy to answer, but it has eluded us:
\begin{conjecture}
The conditions from \ref{CleanCond} are satisfied only when $F(i)=0$.
\end{conjecture}
\bibliographystyle{amsplain}
\bibliography{lit}
\end{document}